\documentclass[12pt]{amsart}

\usepackage{graphicx, amssymb, labelfig}

\newtheorem{thm}{Theorem}
\newtheorem{prop}[thm]{Proposition}
\newtheorem{lem}[thm]{Lemma}

\newtheorem{conj}[thm]{Conjecture}

\newcommand{\SL}{\mathrm{SL}_2(\mathbb{C})}
\newcommand{\db}{\mathbin{/\mkern-6mu/}}
\newcommand{\id}{\ensuremath{\,\mathrm{Id}}}
\newcommand{\Id}{\ensuremath{\,\mathrm{Id}}}
\newcommand{\RR}{\mathcal X_{\SL}}
\newcommand{\SSS}{\mathcal S_A}
\newcommand{\End}{\mathrm{End}}
\newcommand{\Tr}{\mathrm{Tr}\,}

\theoremstyle{remark}

\title[Representations of the skein algebra ]{Representations of the Kauffman skein algebra \\ of small surfaces}

\author{Nurdin Takenov}
\email{greanvert@gmail.com}

\date{\today}
\thanks{This research was partially supported by grants DMS-1105402 and DMS-1406559  from the U.S. National Science Foundation.}

\begin{document}

\begin{abstract} We prove a uniqueness result for finite-dimensional representations of the Kauffman skein algebra $\SSS(S)$ of a surface $S$, when $A$ is a root of unity and when the surface $S$ is a sphere with at most four punctures or a torus with at most one puncture. We show that, if two irreducible representations of $\SSS(S)$ have the same classical shadow and the same puncture invariants, and if this classical shadow is sufficiently generic in the character variety $\RR(S)$, then the two representations are isomorphic. 
\end{abstract}

\maketitle

Let $S$ be an oriented surface (without boundary) with finite topological type. The \emph{Kauffman skein algebra} $\mathcal S_A(S)$ of $S$ is a certain quantization of the character variety
$$
\RR(S) = \{ \text{group homomorphisms }r\colon  \pi_1(S) \to \SL \}\db \SL
$$
depending on a parameter $A=\mathrm e^{\pi\mathrm i \hbar}\in \mathbb C-\{0\}$. The elements of  $\SSS(S)$ are represented by linear combinations of framed links in the thickened surface $S\times [0,1]$, considered modulo certain relations;  see \cite{TuraevPoisson, BFK, PrzSik} and \S \ref{sect:KauffmanAlgebra}. 

In \cite{BonWon1, BonWon2, BonWon3}, Bonahon and Wong consider finite-dimensional representations of the skein algebra $\SSS(S)$ when $A$ is a root of unit,  construct invariants of irreducible representations, and show that any set of invariant is realized. The purpose of the current paper is to provide a uniqueness component to their results when the surface $S$ is small. 

The precise result of \cite{BonWon1} is the following. The setup requires $A^2$ to be an $N$--root of unity with $N$ odd, and we will restrict  attention to the case where $A^N=-1$ as results are easier to state; because $N$ is odd, this case is equivalent to requiring that $A$ be an $N$--root of $-1$. The case $A^N=+1$ can be deduced from the case $A^N=-1$ by using spin structures on the surface; see \cite{Barrett, BonWon1}. 

By definition of the geometric invariant theory quotient involved in the definition of $\RR(S)$, two homomorphisms $r$, $r'\colon \pi_1(S) \to \SL$ define the same element of $\RR(S)$ if and only if they associate the same trace $\Tr r(\gamma) = \Tr r'(\gamma)$ to each element $\gamma \in \pi_1(S)$.

\begin{thm}[\cite{BonWon1}]
\label{thm:InvariantsExist}
Suppose that $A$ is a primitive $N$-root of $-1$ with $N$ odd, and let $\rho \colon \SSS(S) \to \End(V)$ be an irreducible finite-dimensional representation of the Kauffman skein algebra. Let $T_N(x)$ be the $N$-th normalized Chebyshev polynomial  of the first kind, defined by the trigonometric equality that $\cos N\theta = \frac12 T_N(2\cos \theta)$. 

\begin{enumerate}
\item  There exists a unique character $r \in \RR(S)$ such that
$$
T_N \bigl( \rho ([K]) \bigr) =- \bigl( \Tr\, r(K) \bigr) \Id_V
$$
in $\End(V)$ for every framed knot $K\subset S \times [0,1]$ whose projection to $S$ has no crossing and whose framing is vertical. 

\item Let $P_k$ be a small simple loop going around the $k$-th puncture of $S$, and consider it as a knot in $S\times[0,1]$ with vertical framing. Then there exists  a number $p_k\in \mathbb C$ such that
$
\rho\bigl([P_k]\bigr) = p_k \Id_V
$.

\item The number $p_k$ of {\upshape (2)} is related to the character  $r \in \RR(S)$ of 
 {\upshape (1)} by the property that  $T_N (p_k) = -\Tr\, r(P_k)$.
\end{enumerate}
\end{thm}

The character  $r \in \RR(S)$ is the \emph{classical shadow} of the representation $\rho$, while the numbers $p_k$ are its \emph{puncture invariants}. Bonahon and Wong also show in \cite{BonWon2, BonWon3} that every character $r\in \RR(S)$ and set of numbers $p_k \in \mathbb C$ associated to the punctures of $S$ that satisfy the relation (3) of Theorem~\ref{thm:InvariantsExist} can be realized as the set of invariants of an irreducible representation of $\SSS(S)$. 

We address the following conjecture which asserts that, generically, an irreducible representation of $\SSS(S)$ should be determined by its classical shadow and its puncture invariants. 

\begin{conj}[Generic Uniqueness Conjecture]
\label{conj:MainConjecture}
There exists a Zarisky dense open subset $\mathcal U \subset \RR(S)$ such that, if two irreducible representations $\rho$, $\rho '\colon \SSS(S) \to \End(V)$ have the same classical shadow and the same puncture invariants, and if the classical shadow  $r $ belongs to this dense subset $\mathcal U \subset \RR(S)$, then the representations $\rho$ and $\rho'$ are isomorphic. 
\end{conj}

Our main result is the following. 

\begin{thm}
\label{thm:MainThmIntro}
The Generic Uniqueness Conjecture~{\upshape \ref{conj:MainConjecture}} holds when the surface $S$ is a sphere with at most $4$ punctures, or a torus with $0$ or $1$ puncture. 
\end{thm}

The cases of the sphere with at most 3 punctures are essentially trivial; see \S \ref{sect:3PuncSphere}. When $S$ is the one-puncture torus, we rely in \S \ref{sect:PuncTorus} on a presentation of $\SSS(S)$ given by Bullock-Przytycki \cite{bullockp} and on the systematic classification of all irreducible representations of the corresponding algebra by Havl\'i\v cek-Po\v sta \cite{havlicekp}; it is however non-trivial to relate the invariants of \cite{havlicekp} to those of Theorem~\ref{thm:InvariantsExist}.  The case of the unpunctured torus follows from the one-puncture case; see \S \ref{sect:NoPuncTorus}. 
 For the four-puncture sphere, we again use in \S \ref{sect:4PuncSphere} a presentation of $\SSS(S)$ exhibited in \cite{bullockp}, and extend to this context the arguments that we had developed for the one-puncture torus; however, the situation is  very significantly more complicated than for the one-puncture torus.
 
 The analysis of \cite{havlicekp} provides many different representations of the skein algebra of the one-puncture torus whose classical shadow is the trivial character. As a consequence, Conjecture~\ref{conj:MainConjecture} cannot hold without the genericity hypothesis.

\section{The Kauffman skein algebra}
\label{sect:KauffmanAlgebra}

Let $S$ be an oriented surface of finite topological type, without boundary. This means that $S = S_{g,p}$ is obtained by removing $p$ points from a closed oriented surface $\bar{S}$ of genus $g$. We consider \textit{framed links} in the thickened surface $S \times [0, 1]$, namely unoriented 1-dimensional submanifolds $K\subset S \times [0, 1]$ endowed with a continuous choice of a vector transverse to $K$ at each point of $K$. A \textit{framed knot} is a  framed link that is connected.

The following definition provides a convenient way to describe a framing, in particular when representing a link by a picture. If the projection of $K\subset S \times [0, 1]$ to $S$ is an immersion, the \textit{vertical framing} for $K$ is the framing that is everywhere  parallel to the $[0, 1]$ factor and points towards 1.

The \textit{framed link algebra} $ \mathcal{K}(S)$ is the vector space over $\mathbb{C}$ freely generated by the isotopy classes of all framed links $K\subset S \times [0, 1]$. This vector space $ \mathcal{K}(S)$ can
be endowed with a multiplication, where the product of $K_1$ and $K_2$ is defined as the framed link $K_1\cdot K_2 \subset S \times [0, 1]$ that is the union of $K_1$ rescaled in $S \times [0, \frac{1}{2}]$ and $K_2$ rescaled in $S \times [\frac{1}{2},1]$. In other words, the product $K_1 \cdot K_2$ is defined by superposition of the framed links $K_1$ and $K_2$. This \textit{superposition operation} is compatible with isotopies, and therefore provides a well-defined algebra structure on $ \mathcal{K}(S)$.

Three framed links $K_1$, $K_0$ and $K_{\infty}$ in $S \times [0, 1]$ form a \textit{Kauffman triple} if the only place where they differ is above a small disk in $S$, where they are as represented in Figure \ref{fig:skein} (as seen from above) and where the framing is vertical and pointing upwards (namely the framing is parallel to the $[0, 1]$ factor and points towards 1).

\begin{figure}[h]
 \SetLabels
( .16* -.1) $K_1 $ \\
(.52 * -.1) $K_0 $ \\
( .86* -.1) $K_\infty $ \\
\endSetLabels
\centerline{\AffixLabels{    \includegraphics[scale=0.5]{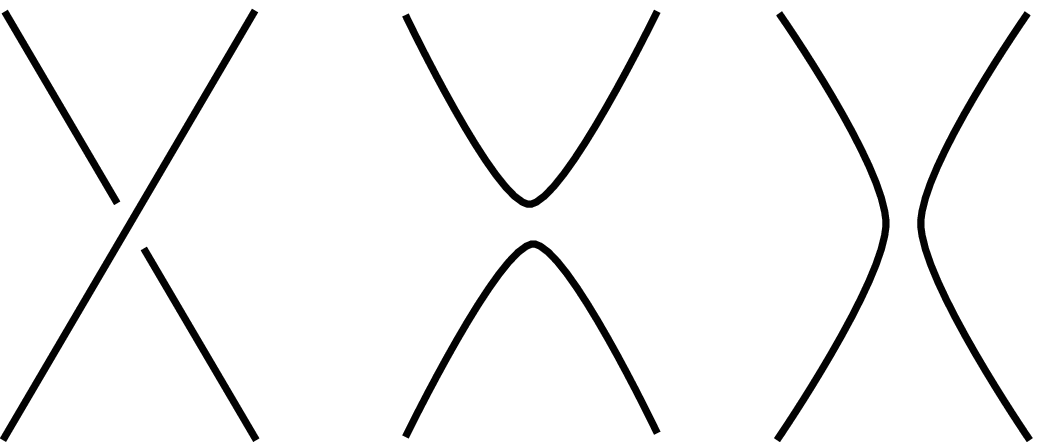} }}
\vskip 5pt
    \caption{A Kauffman triple}
    \label{fig:skein}
\end{figure}

For $A \in \mathbb{C} - \{0\}$, the \textit{Kauffman skein algebra} $\mathcal{S}_A(S)$ is the quotient of the
framed link algebra $ \mathcal{K}(S)$ by the linear subspace generated by:
\begin{enumerate}
\item all elements $K_1 - A^{-1} K_0 - A K_{\infty}\in \mathcal{K}(S)$ as $(K_1, K_0, K_{\infty})$ ranges over all Kauffman triples;
\item the element $\bigcirc + (A^2+A^{-2}) \varnothing $ where $\bigcirc \subset S\times [0,1]$ is an unknot projecting to asimple loop in $S$ and endowed with the vertical framing, and where $\varnothing$ is the empty link. 
\end{enumerate}

The superposition operation descends to a multiplication in  $\SSS(S)$, and endows  $\mathcal{S}_A(S)$ with the structure of an algebra. The class $[ \varnothing ]$ of the empty link is an identity element in  $\mathcal{S}_A(S)$, and we usually identify it to the scalar 1.

An element $[K] \in \mathcal{S}_A(S)$ represented by a framed link $K\subset S \times [0, 1]$ is a \emph{skein} in $S$. 

Throughout the article, we will assume that the parameter $A$ is a primitive $N$--root of $-1$ with $N$ odd; for instance, $A= \mathrm e^{ {\pi\mathrm i}/N}$. Because $N$ is odd, this implies that $A^2$ and $A^4$ are primitive $N$--roots of unity, a property that is frequently used in the article.

\section{Chebyshev polynomials}

The \emph{normalized $n$-th Chebyshev polynomial of the first kind} is the polynomial $T_n(x)$ such that $\Tr M^n = T_n(\Tr M)$ for every $M\in \SL$. It can be inductively computed by the recurrence relation that $T_n(x) = x T_{n-1}(x) - T_{n-2}(x)$, combined with the initial conditions $T_0(x)=2$ and $T_1(x)=x$. 

The polynomial $T_n(x)$ contains only even powers of $x$ when $n$ is even,  and only odd powers of $x$ when $n$  is odd. Also, applying the relation $\Tr M^n = T_n(\Tr M)$ to a rotation matrix gives the trigonometric identity that $2 \cos n\theta = T_n(2 \cos \theta)$. 

We will use the following computations.

\begin{lem}
\label{lem:chebyshev}
$ $
\begin{enumerate}
\item[(a)]
$T_n(a+a^{-1}) = a^n + a^{-n}$;
\item[(b)]
$T_N(x) - a^N - a^{-N} = \prod_{k=1}^{N}\left(x - a A^{2k}- a^{-1} A^{-2k}\right)$.
\end{enumerate}
\end{lem}
\begin{proof} Property (a) is a simple application of the relation $\Tr M^n = T_n(\Tr M)$ to a matrix $M\in \SL$ with eigenvalues $a$, $a^{-1}$. 

For (b), note that (a) implies that the solutions of the equation $T_N(x) = a^N +a^{-N}$ are the numbers $x= b +b^{-1}$ with $b^N = a^N$. Since $A^2$ is a primitive $N$-root of unity, these are the numbers of the form $x= aA^{2k} +a^{-1} A^{-2k}$. The equality then follows from the fact that the highest degree term of $T_N(x)$ is $x^N$. 
\end{proof}

\section{The sphere with at most three punctures}
\label{sect:3PuncSphere}

This case is essentially trivial. Indeed, in the sphere with at most 3 punctures, every simple closed curve is isotopic, either to a trivial knot, or to  a simple loop $P_k$ going around one of the punctures. It follows that the algebra $\SSS(S)$ is generated by the skeins $[P_k]$, so that every irreducible representation $\rho \colon \SSS(S) \to \End(V)$ is completely determined by its puncture invariants $p_k$ (which also determine the classical shadow).

\section{The one-puncture torus $S_{1,1}$}
\label{sect:PuncTorus}

This section is devoted to the case of the one-puncture torus $S_{1,1}$. 

\subsection{A presentation for the skein algebra $\SSS(S_{1,1})$}
\label{subsect:PresentationPunctTorus}
 We will use the presentation of $\SSS(S_{1,1})$ given by  Bullock and Przytycki in \cite{bullockp}. 

Identify the punctured torus $S_{1,1}$ to the quotient space obtained from a square $[0,1] \times [0,1]$ by removing the four corners  and gluing together opposite sides in the usual fashion. Let $X_1$ be the closed curve that is the image in $S_{1,1}$ of the vertical line segment $\{\frac12\} \times [0,1]$, let $X_2$ be the  image of the horizontal line segment $  [0,1] \times \{\frac12\}$, and let $X_3$ be the closed curve that is the image of the two slope 1 line segments respectively parametrized by $t\mapsto (t+\frac12, t)$ and $t \mapsto (t, t+\frac12)$ with $0 \leq t\leq \frac12$. See Figure~\ref{fig:PunctTorus}. 

Also, let $P\subset S_{1,1}$ be a simple loop going around the puncture.

\begin{figure}[htbp]

\SetLabels
( .17* .75) $X_1 $ \\
( .78* .79) $X_2 $ \\
( .15*.3 ) $ X_3$ \\
(.27 *.15 ) $ X_3$ \\
(.7 * .34) $P $ \\
( .87*.34 ) $P $ \\
( .7* .1) $ P$ \\
(.87 *.1 ) $ P$ \\
\endSetLabels
\centerline{\AffixLabels{\includegraphics{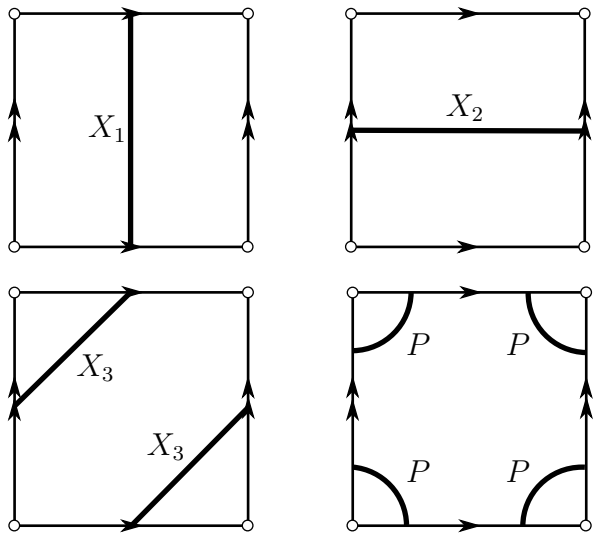}}}
\caption{Curves on the one-puncture torus}
\label{fig:PunctTorus}
\end{figure}

\begin{prop}[\cite{bullockp}]
\label{prop:PresentationPuncTorus}

The Kauffman skein algebra $\SSS(S_{1,1})$ admits a presentation by  generators $X_1$, $X_2$ and $X_3$ and relations
\begin{align*}
    A X_1 X_{2}-A^{-1} X_{2} X_{1}&=(A^2-A^{-2})X_{3} \\
    A X_2 X_{3}-A^{-1} X_{3} X_{2}&=(A^2-A^{-2})X_{1} \\
    A X_3 X_{1}-A^{-1} X_{1} X_{3}&=(A^2-A^{-2})X_{2}
\end{align*}
where the $X_i$ are represented by the closed curves indicated above, endowed with vertical framing. 

If the  loop $P$ going around the puncture is also endowed with the vertical framing, the corresponding element of $\SSS(S_{1,1})$ is equal to
$$
P=A X_1 X_2 X_3 - A^2 X_1^2 - A^{-2}X_2^2 - A^2 X_3^2 +  A^2 + A^{-2}.
$$ This element is central in $\SSS(S_{1,1})$. \qed
\end{prop}

The above presentation is also a presentation of another algebra, the algebra  $\mathrm U'_q(\mathfrak{so}_3)$ whose irreducible representations were classified by Havl\'{i}\v{c}ek and Po\v{s}ta \cite{havlicekp}. We will heavily rely on the arguments of \cite{havlicekp}, while adapting them to our goals.

\subsection{Reconstructing an irreducible representation of $\SSS(S_{1,1})$}

\begin{thm}
\label{thm:ReconstructPuncTorus}
Let $\rho \colon \mathcal{S}_A(S_{1,1}) \to \text{End}(V)$ be  an irreducible representation with classical shadow  $r\in \RR(S_{1,1})$ and puncture invariant $p\in \mathbb C$. Suppose that 
\begin{align*}
&\Tr r(X_3) \neq \pm2\\
\text{and } &\Tr r(X_1)\Tr r(X_2)\Tr r(X_3) + \Tr r(X_1)^2 + \Tr r(X_2)^2 \neq 0
\end{align*}
for the curves $X_1$, $X_2$, $X_3$ of \S {\upshape \ref{subsect:PresentationPunctTorus}}. Then, up to isomorphism, the representation $\rho \colon \mathcal{S}_A(S_{1,1}) \to \text{End}(V)$ is completely determined by $r$ and $p$. In addition, $V$ has dimension $N$.  
\end{thm}

Theorem~\ref{thm:ReconstructPuncTorus} proves our Generic Uniqueness Theorem~\ref{thm:MainThmIntro} in the case of the one-puncture torus $S_{1,1}$, since its hypotheses describe a Zarisky dense open subset of the character variety $\RR(S)$. We will late provide a slightly more general result, Theorem~\ref{thm:ReconstructPuncTorusBis}. 

The proof of Theorem~\ref{thm:ReconstructPuncTorus} will take a while, and we will split it into several lemmas. We assume that the hypotheses of Theorem~\ref{thm:ReconstructPuncTorus}  hold, for the remainder of this section. 

For notational convenience, set $t_i = -\Tr r(X_i)$ for $i=1$, $2$, $3$, so that 
$T_N\bigl(\rho(X_i) \bigr) = t_i \Id_V$.The hypotheses of Theorem~\ref{thm:ReconstructPuncTorus} are then that $t_3 \neq \pm2$ and $t_1t_2t_3 + t_1^2 +t_2^2 \neq 0$.

The numbers $t_1$, $t_2$, $t_3$ and the puncture invariant $p$ are related by the following equation.

\begin{lem}
\label{lem:RelationsPunctureChebyshev}
$$
T_N(p) = - t_1 t_2 t_3 - t_1^2 - t_2^2 - t_3^2 + 2
$$
\end{lem}

\begin{proof}
This is a consequence of the special case where $N=1$ and $A=-1$. 
Indeed, an observation of Bullock \cite{Bullock} (see also \cite{PrzSik}) shows that the character $r\in \RR(S_{1,1})$ uniquely determines a homomorphism $\Theta_{r} \colon \mathcal S_{-1}(S_{1,1}) \to \mathbb C$ by the property that $\Theta_{r} \bigl( [K] \bigr) = - \Tr r(K)$ for every framed knot $K\subset S_{1,1} \times [0,1]$. 

In the special case $A=-1$ considered, the second half of Proposition~\ref{prop:PresentationPuncTorus} states that the elements $X_1$, $X_2$, $X_3$, $P \in \mathcal S_{-1}(S_{1,1})$ satisfy the relation
$$
P=- X_1 X_2 X_3 -  X_1^2 - X_2^2 -  X_3^2 +  2
$$ 
in $\mathcal S_{-1}(S_{1,1})$. Applying the homomorphism $\Theta_{r}$ on both sides of the equation, and remembering that $\Theta_{r} (X_i) = - \Tr r(X_i) = t_i$ and $\Theta_{r} (P) = - \Tr r(P) =T_N(p)$, then provides the equality sought.
\end{proof}

Choose a number  $x_3 \in \mathbb{C}$ such that $t_3 = x_3^N + x_3^{-N}$. 

Because $T_N\bigl (\rho(X_3)\bigr) = t_3 \Id_V$, all possible eigenvalues $\lambda$ of  $\rho(X_3)$ are such that $T_N (\lambda) = t_3=x_3^N + x_3^{-N} $, and  therefore equal to one of the numbers  $\lambda_k=x_3 A^{2k}+ x_3^{-1} A^{-2k}$ by Lemma~\ref{lem:chebyshev}(a). Since $t_3 \neq \pm 2$ by hypothesis, the numbers  $\lambda_k$ with $k=1$, $2$,\dots, $N$ are distinct. It is convenient to consider all integer indices $k$, so that $\lambda_{k+N} = \lambda_k$. 

Define $V_k=\{v\in V; \rho(X_3)v = \lambda_k v \}$. Namely, $V_k$ is the eigenspace corresponding to $\lambda_k$ if $\lambda_k$ is really an eigenvalue of $\rho(X_3)$, and is 0 otherwise.

Our key tool is provided by  the operators $\mathcal{U}_k$ and $\mathcal{D}_k\in \End(V)$ defined by
\begin{align*}
\mathcal{U}_k & = A \rho(X_1) - x_3 A^{2k} \rho(X_2)\\
\mathcal{D}_k & = A \rho(X_1) - x_3^{-1} A^{-2k} \rho(X_2).
\end{align*}

These ``up'' and ``down'' operators $\mathcal{U}_k$ and $\mathcal{D}_k$ are borrowed from \cite{havlicekp}. What makes them  so useful is the following property. 

\begin{lem}
\label{lem:UpDown}
The operator $\mathcal U_k$ sends the subspace $V_k$ to $V_{k+1}$, and $\mathcal D_k$ sends $V_k$ to $V_{k-1}$. 
\end{lem}

\begin{proof}We want  to show that $\rho(X_3) \mathcal{U}_k v_k = \lambda_{k+1} \mathcal{U}_k v_k$ for every $v_k\in V_k$. For this, expand
$$
\rho(X_3) \mathcal{U}_k v_k=\rho(X_3)\left(A \rho(X_1) - x_3 A^{2k} \rho(X_2)\right) v_k = \left(A \rho(X_3 X_1) - x_3 A^{2k} \rho(X_3 X_2)\right) v_k .
$$

Using the relations of Proposition~\ref{prop:PresentationPuncTorus}, 
\begin{align*}
\rho(X_3) \mathcal{U}_k v_k & =(A \rho(X_3 X_1) - x_3 A^{2k} \rho(X_3 X_2)) v_k  \\
& =  A \rho( A^{-2} X_1 X_3 + A^{-1}(A^2 - A^{-2}) X_2)v_k \\
&\qquad \qquad- x_3 A^{2k} \rho(A^2 X_2 X_3 - A (A^2 - A^{-2}) X_1) v_k \\
& =  A^{-1} \rho(X_1)\rho(X_3)v_k + (A^2-A^{-2})\rho(X_2)v_k  \\
&\qquad \qquad -  x_3 A^{2k+2} \rho(X_2)\rho(X_3)v_k + x_3 A^{2k+1} (A^2 - A^{-2}) \rho(X_1) v_k.
\end{align*}

Using the facts that $\rho(X_3) v_k = \lambda_k v_k$ and $\lambda_k=x_3 A^{2k}+ x_3^{-1} A^{-2k}$, we obtain
\begin{align*}
\rho(X_3) \mathcal{U}_k v_k &  =  A^{-1} \rho(X_1)\lambda_kv_k + (A^2-A^{-2})\rho(X_2)v_k \\
&\qquad \qquad - x_3 A^{2k+2} \rho(X_2)\lambda_k v_k+ x_3 A^{2k+1} (A^2 - A^{-2}) \rho(X_1) v_k  \\
& =  (x_3 A^{2k+3} + x_3^{-1} A^{-2k-1}) \rho(X_1)v_k - (x_3^2 A^{4k+2} + A^{-2}) \rho(X_2) v_k \\
& = \lambda_{k+1} \bigl(A \rho(X_1) - x_3 A^{2k} \rho(X_2) \bigr) v_k = \lambda_{k+1} \mathcal{U}_k v_k
\end{align*}
which shows that $\mathcal{U}_k v_k$ belongs to $V_{k+1}$. 

The proof that $\mathcal{D}_k$ sends $V_k$ to $V_{k-1}$ is similar. 
\end{proof}

The indexing of the eigenspaces $V_k$ of $\rho(X_3)$ and of the operators $\mathcal{U}_k$, $\mathcal{D}_k$ depends on our choice of $x_3$ such that $t_3 = x_3^N + x_3^{-N}$. In particular, replacing $x_3$ by $x_3 A^{2l}$ replaces $V_k$, $\mathcal{U}_k$, $\mathcal{D}_k$ by $V_{k - l }$, $\mathcal{U}_{k - l}$ and $\mathcal{D}_{k - l}$, respectively. Similarly, replacing $x_3$ by $x_3^{-1}$ flips the order and replaces  $V_k$, $\mathcal{U}_k$, $\mathcal{D}_k$ by  $V_{N-k}$, $\mathcal{D}_{N-k}$, $\mathcal{U}_{N-k}$.

\begin{lem}
\label{lem:ExpressRhoUpDown}
For every  vector $v_k \in V_k$, 
\begin{align*}
\rho(X_1) v_k &= -\frac{ x_3^{-1} A^{-2k-1}}{x_3 A^{2k} - x_3^{-1} A^{-2k}} \mathcal{U}_k v_k + \frac{ x_3 A^{2k-1}}{x_3 A^{2k} - x_3^{-1} A^{-2k}} \mathcal{D}_k v_k \\
\rho(X_2) v_k &= -\frac{ 1 }{x_3 A^{2k} - x_3^{-1} A^{-2k}} \mathcal{U}_k v_k + \frac{ 1 }{x_3 A^{2k} - x_3^{-1} A^{-2k}} \mathcal{D}_k v_k \\
\rho(X_3) v_k &= (x_3 A^{2k}+ x_3^{-1} A^{-2k}) v_k.
\end{align*}
\end{lem}
\begin{proof} The first two lines are obtained by solving for  $\rho(X_1) v_k $ and $\rho(X_3) v_k $ in the definition of $\mathcal U_k$ and $\mathcal D_k$. We just need to check that the denominators are non-zero. But this immediately follows from the hypothesis that $t_3 = x_3^N + x_3^{-N}$ is different from $\pm2$. 

The last equation is just a rephrasing of the definition of the subspace $V_k$. 
\end{proof}

Our next computation shows how the operators $\mathcal{U}_k$ and $\mathcal{D}_k$ interact with each other. Note that $\mathcal U_k$ sends $V_k$ to $V_{k+1}$, and that $\mathcal D_{k+1}$ sends $V_{k+1}$ back to $V_k$. 

\begin{lem}
\label{lem:ComposeUpAndDown} For every $v_k \in V_k$, 
$$
\mathcal{D}_{k+1} \mathcal{U}_k v_k = - (p+ x_3^2 A^{4k+2} + x_3^{-2} A^{-4k-2}) v_k.
$$
\end{lem}

\begin{proof} We begin by expanding
\begin{align*}
\mathcal{D}_{k+1} \mathcal{U}_k v_k &= \bigl (A \rho(X_1) - x_3^{-1} A^{-2k-2} \rho(X_2) \bigr) \bigl(A \rho(X_1) - x_3 A^{2k} \rho(X_2) \bigr)v_k \\
&= \bigl(A^2 \rho(X_1^2) - x_3^{-1} A^{-2k-1} \rho(X_2 X_1) - x_3 A^{2k+1} \rho(X_1 X_2) + A^{-2} \rho(X_2^2) \bigr)v_k.
\end{align*}

The puncture invariant $p$ is defined by the property that $\rho(P)=p \Id_V$ for the puncture element 
$$
P=A X_1 X_2 X_3 - A^2 X_1^2 - A^{-2}X_2^2 - A^2 X_3^2 +  A^2 + A^{-2}.
$$
Therefore, using the fact that $\rho(X_3)v_k = \lambda_k v_k = (x_3A^{2k} + x_3^{-1}A^{-2k})v_k$ and the relation that $AX_1X_2 - A^{-1}X_2X_1 = (A^2 - A^{-2})X_3$, 
\begin{align*}
\mathcal{D}_{k+1} \mathcal{U}_k v_k + p v_k & = A^2 \rho(X_1^2)v_k - x_3^{-1} A^{-2k-1} \rho(X_2 X_1)v_k - x_3 A^{2k+1} \rho(X_1 X_2) v_k\\
& \qquad \qquad + A^{-2} \rho(X_2^2)v_k +  A \rho(X_1 X_2 X_3)v_k - A^2 \rho(X_1^2)v_k \\
& \qquad \qquad - A^{-2}\rho(X_2^2)v_k  - A^2 \rho(X_3^2)v_k +  (A^2 + A^{-2})v_k  \\
& =-(  x_3^2 A^{4k+2} + x_3^{-2} A^{-4k-2} )v_k 
\end{align*}
after simplifications. This concludes the proof.
\end{proof} 

\begin{lem}
\label{lem:ProductUpsAndDowns}
Consider the map
$$
\prod_{j=1}^N\mathcal{D}_{k+j} \prod_{j=1}^N\mathcal{U}_{k+N-j} =  \mathcal D_{k+1}\mathcal D_{k+2} \dots \mathcal D_{k+N-1}\mathcal D_{k+N} \mathcal U_{k+N-1}\mathcal U_{k+N-2} \dots \mathcal U_{k+1}\mathcal U_{k}.
$$
For every $v_k\in V_k$,
 $$
\prod_{j=1}^N\mathcal{D}_{k+j} \prod_{j=1}^N\mathcal{U}_{k+N-j} v_k = ( t_1 t_2 t_3 + t_1^2+  t_2^2 )  v_k
$$
\end{lem}

\begin{proof} Note that $\mathcal U_{k+j-1}\mathcal U_{k+j-2} \dots \mathcal U_{k+1}\mathcal U_{k}v_k$ is an element of $V_j$ for every $j$. By successive applications of Lemma~\ref{lem:ComposeUpAndDown}, we conclude that
$$
\prod_{j=1}^N\mathcal{D}_{k+j} \prod_{j=1}^N\mathcal{U}_{k+N-j} v_k = -  \prod_{j=1}^N (p + x_3^2 A^{4k+4j+2} + x_3^{-2} A^{-4k-4j-2}) v_k.
$$
Because $A^2$ and $A^4$ are primitive $N$-roots of unity, the set of powers $A^{4k+4j}$ with $j=1$, $2$, \dots, $N$ is the same as the set of powers $A^{2l}$ with $l=1$, $2$, \dots, $N$. Therefore
\begin{align*}
 \prod_{j=1}^N (p + x_3^2 A^{4k+4j+2} + x_3^{-2} A^{-4k-4j-2}) 
  &=  \prod_{l=1}^N (p + x_3^2 A^{2l} + x_3^{-2} A^{-2l}) \\
 &= T_N(p) + x_3^{2N} + x_3^{-2N} \\
 &= -t_1 t_2 t_3 -t_1^2 - t_2^2
\end{align*}
by Lemmas~\ref{lem:chebyshev}(b) and \ref{lem:RelationsPunctureChebyshev}, remembering that $t_3 = x_3^N + x_3^{-N}$. This completes the proof.
\end{proof}

We now use the fact that the representation $\rho \colon \SSS(S_{1,1}) \to \End(V)$ is irreducible. 

\begin{lem}
\label{lem:Ndimensional}
The space $V$ has dimension $N$, and all eigenspaces $V_k$ of $\rho(X_3)$ are $1$-dimensional. 

More precisely, $V$ admits a basis $\{ v_1, v_2, \dots, v_N\}$ where each $v_k$ generates the eigenspace $V_k$, and where for some $u\neq 0$
$$
\mathcal U_k v_k =
\begin{cases}
v_{k+1} &\text{ if } 1\leq k \leq N-1\\
u v_1 &\text{ if } k=N
\end{cases}
$$
and
$$
\mathcal D_k v_k =
\begin{cases}
-\frac1u (p+ x_3^2 A^{2} + x_3^{-2} A^{-2})  v_N &\text{ if } k=1\\
- (p+ x_3^2 A^{4k-2} + x_3^{-2} A^{-4k+2}) v_{k-1} &\text{ if } 2\leq k \leq N.
\end{cases}
$$

\end{lem}

\begin{proof} The operator $\rho(X_3)\in \End(V)$ admits at least one non-zero eigenvector. Therefore, some $V_{k_0}$ is different from 0.

The map $\prod_{j=1}^N\mathcal{U}_{k_0+N-j} = \mathcal U_{k_0+N-1}\mathcal U_{k_0+N-2} \dots \mathcal U_{k_0+1}\mathcal U_{k_0}$ of Lemma~\ref{lem:ProductUpsAndDowns} sends $V_{k_0}$ to $V_{k_0}$. Because  $t_1 t_2 t_3 + (t_1)^2+  (t_2)^2 \neq 0$ by hypothesis in Theorem~\ref{thm:ReconstructPuncTorus}, Lemma~\ref{lem:ProductUpsAndDowns} shows that this map is not the 0 map on $V_{k_0}$. Therefore, there is an eigenvalue $u\neq 0$ and a non-zero eigenvector $v_{k_0} \in V_{k_0}$ such that
$$
 \mathcal U_{k_0+N-1}\mathcal U_{k_0+N-2} \dots \mathcal U_{k_0+1}\mathcal U_{k_0} v_{k_0} = uv_{k_0}.
$$

We can now arrange that $k_0=1$. Indeed, if we apply  $\mathcal U_{N}\mathcal U_{N-1} \dots \mathcal U_{k_0+1}\mathcal U_{k_0}  $ to both sides of the above equality and if we set $v_1 = \mathcal U_{N}\mathcal U_{N-1} \dots \mathcal U_{k_0+1}\mathcal U_{k_0} v_{k_0} \in V_{N+1} = V_1$, we obtain that
$$
 \mathcal U_{N}\mathcal U_{N-1} \dots \mathcal U_{2}\mathcal U_{1} v_{1} = uv_{1}.
$$

Then, set $v_k =  \mathcal U_{k-1}\mathcal U_{k-2} \dots \mathcal U_{2}\mathcal U_{1} v_{1} \in V_k$ for every $k=1$, $2$, \dots, $N$. Let $W$ be the $N$-dimensional linear subspace of $V$ spanned by the vectors $v_1$, $v_2$, \dots, $v_N$.

By construction, 
$$
\mathcal U_k v_k =
\begin{cases}
v_{k+1} &\text{ if } 1\leq k \leq N-1\\
u v_1 &\text{ if } k=N.
\end{cases}
$$
Also, by Lemma~\ref{lem:ComposeUpAndDown},
$$
\mathcal D_k v_k =
\begin{cases}
-\frac1u (p+ x_3^2 A^{2} + x_3^{-2} A^{-2})  v_N &\text{ if } k=1\\
- (p+ x_3^2 A^{4k-2} + x_3^{-2} A^{-4k+2}) v_{k-1} &\text{ if } 2\leq k \leq N.
\end{cases}
$$

An immediate consequence of these formulas is that, for every $k$, $\mathcal U_k v_k$ and $\mathcal D_k v_k $ both belong to $W$. Lemma~\ref{lem:ExpressRhoUpDown} then shows that $\rho(X_1)v_k$,  $\rho(X_2)v_k$ and $\rho(X_3)v_k$ belong to $W$.  This proves that $W$ is invariant under the image of the representation $\rho \colon \SSS(S_{1,1}) \to \End(V)$. By irreducibility of $\rho$, it follows that $W=V$. Therefore, $V$ has dimension $N$, and each eigenspace $V_k$ of $\rho(X_3)$ is the line generated by the vector $v_k$. 
\end{proof}

Combining the above formulas for $\mathcal U_k v_k$ and $\mathcal D_k v_k$ with Lemma~\ref{lem:ExpressRhoUpDown} completely determines $\rho(X_1)$ and $\rho(X_2)$ in terms of the parameters $p$, $x_3$ and $u$. Here, $p$ and $x_3$ are given by the data of Theorem~\ref{thm:ReconstructPuncTorus}. We now determine $u$ in terms of the rest of this data. 

\begin{lem}
\label{lem:ComputeU}
$$
u =  -t_1 - x_3^N t_2.
$$
\end{lem}
\begin{proof} 
We need to relate the number $u$ to the numbers $t_i$ such that $T_N \bigl( \rho (X_i) \bigr) = t_i \Id_V$. 

By Lemma~\ref{lem:ExpressRhoUpDown}, $\rho(X_1)$ can be decomposed as $\rho(X_1) = \mathcal U + \mathcal D$ where, in the basis $\{v_1, v_2, \dots, v_N\}$ provided by Lemma~\ref{lem:Ndimensional}, 
\begin{align*}
\mathcal{U} v_k & = -\frac{ x_3^{-1} A^{-2k-1}}{x_3 A^{2k} - x_3^{-1} A^{-2k}} \mathcal{U}_k v_k \\
\mathcal{D} v_k & = \frac{ x_3 A^{2k-1}}{x_3 A^{2k} - x_3^{-1} A^{-2k}} \mathcal{D}_k v_k.
\end{align*}
 The crucial property is that $\mathcal{U}$ and $\mathcal{D}$ cyclically permute the eigenspaces $V_k$ in opposite directions, namely $\mathcal{U}$ sends each $V_k$ to $V_{k+1}$ while $\mathcal{D}$ sends $V_k$ to $V_{k-1}$.
 
 Expanding $T_N \bigl(\rho(X_1) \bigr)=T_N(\mathcal{U}+\mathcal{D})$ gives a linear combination of terms of the form $A_1 A_2 \dots A_m$, where each $A_i$ is either $\mathcal{U}$ or $\mathcal{D}$, and where $m\leq N$ is odd. If such a monomial contains $n$ terms that are $\mathcal{U}$ and $m-n$ terms that are $\mathcal{D}$,  it sends $V_k$ to $V_{k+2n-m}$. Therefore, since $T_N\bigl(\rho(X_1)\bigr) = t_1 \Id_V$, this expansion of $T_N(\mathcal{U}+\mathcal{D})$ contains only terms such that  $2n-m = 0 \mod N$.  Since $m$ is odd and $0\le n \le m\le N$, we only have  two possibilities: $(m,n) = (N,0)$ or $(m,n) = (N,N)$. Since the highest degree term of $T_N(x)$ is $x^N$, this proves that 
 $$
 T_N \bigl(\rho(X_1) \bigr)=T_N(\mathcal{U}+\mathcal{D})= \mathcal U^N + \mathcal D^N.
 $$
 
 As $A^2$ is a primitive $N$-root of unity, $\prod_{k=1}^N (x_3 A^{2k} - x_3^{-1} A^{-2k})= x_3^{N}-x_3^{-N}$. It follows that
\begin{align*}
\mathcal U^N v_k &= \frac{x_3^{-N}}{x_3^{N}-x_3^{-N}} \mathcal U_{k+N-1}\mathcal U_{k+N-2} \dots \mathcal U_{k+1}\mathcal U_{k} v_k\\
&= \frac{ux_3^{-N}}{x_3^{N}-x_3^{-N}}  v_k
\end{align*}
using the expression for $\mathcal U_l v_l$ given by Lemma~\ref{lem:Ndimensional}. 

Similarly,
\begin{align*}
\mathcal D^N v_k &= -\frac{x_3^{N}}{x_3^{N}-x_3^{-N}} \mathcal D_{k-N+1}\mathcal D_{k-N+2} \dots \mathcal D_{k-1}\mathcal D_{k} v_k\\
&= +\frac{x_3^{N}}{x_3^{N}-x_3^{-N}}  \frac1u  \prod_{j=1}^N (p+ x_3^2 A^{4k-4j+2} + x_3^{-2} A^{-4k+4j-2}) v_k\\
&=  - \frac{u^{-1} x_3^{N}}{x_3^{N}-x_3^{-N}} ( t_1 t_2 t_3 + t_1^2+  t_2^2 ) v_k
\end{align*}
where the last equality is proved by the same computation as in the proof of Lemma~\ref{lem:ProductUpsAndDowns}. 

Comparing the equalities $T_N\bigl(\rho(X_1)\bigr) = t_1 \Id_V$ and $ T_N \bigl(\rho(X_1) \bigr)= \mathcal U^N + \mathcal D^N$, it follows that 
$$
t_1 = - \frac{ux_3^{-N}}{x_3^{N}-x_3^{-N}} - \frac{u^{-1}x_3^{N}}{x_3^{N}-x_3^{-N}} ( t_1 t_2 t_3 + t_1^2+  t_2^2 ).
$$

This almost determines $u$, up to two possibilities. To resolve the ambiguity, we perform similar computations for $T_N\bigl(\rho(X_2)\bigr) = t_2 \Id_V$, which give 
$$
t_2 =  \frac{u}{x_3^{N}-x_3^{-N}} + \frac{u^{-1}}{x_3^{N}-x_3^{-N}} ( t_1 t_2 t_3 + t_1^2+  t_2^2 ).
$$

Combining these expressions for $t_1$ and $t_2$ shows that $u = -t_1 - t_2 x_3^N$. 
\end{proof}

We now just need to combine Lemmas~\ref{lem:ExpressRhoUpDown}, \ref{lem:Ndimensional} and \ref{lem:ComputeU} to obtain: 

\begin{lem}
\label{lem:PunctTorusExplicitFormula}
\pushQED{\qed}
\begin{align*}
\rho(X_1) v_k &=
\begin{cases}
 -\frac{ x_3^{-1} A^{-3}}{x_3 A^{2} - x_3^{-1} A^{-2}} v_{2} - \frac{ x_3 A(p+ x_3^2 A^{2} + x_3^{-2} A^{-2})}{(x_3 A^{2} - x_3^{-1} A^{-2})(t_1 + t_2 x_3^N)}    v_N &\text{ if } k=1
 \\
 -\frac{ x_3^{-1} A^{-2k-1}}{x_3 A^{2k} - x_3^{-1} A^{-2k}} v_{k+1}  + \frac{ x_3 A^{2k-1}(p+ x_3^2 A^{4k-2} + x_3^{-2} A^{-4k+2})}{x_3 A^{2k} - x_3^{-1} A^{-2k}}v_{k-1} &\text{ if } 2\leq k \leq N-1
 \\
 \frac{ x_3^{-1} A^{-1}(t_1 + t_2 x_3^N)}{x_3  - x_3^{-1} }v_1 + \frac{ x_3 A^{-1}(p+ x_3^2 A^{-2} + x_3^{-2} A^{2})}{x_3  - x_3^{-1} }v_{k-1}  &\text{ if } k=N\\
\end{cases}
 \\
\rho(X_2) v_k &=
\begin{cases}
 -\frac{ 1}{x_3 A^{2} - x_3^{-1} A^{-2}} v_{2} - \frac{ p+ x_3^2 A^{2} + x_3^{-2} A^{-2}}{(x_3 A^{2} - x_3^{-1} A^{-2})(t_1 + t_2 x_3^N)}    v_N &\text{ if } k=1
 \\
 -\frac{ 1}{x_3 A^{2k} - x_3^{-1} A^{-2k}} v_{k+1}  + \frac{ p+ x_3^2 A^{4k-2} + x_3^{-2} A^{-4k+2}}{x_3 A^{2k} - x_3^{-1} A^{-2k}}v_{k-1} &\text{ if } 2\leq k \leq N-1
 \\
 \frac{ t_1 + t_2 x_3^N}{x_3  - x_3^{-1} }v_1 + \frac{ p+ x_3^2 A^{-2} + x_3^{-2} A^{+2}}{x_3  - x_3^{-1} }v_{k-1}  &\text{ if } k=N\\
\end{cases}
\\
\rho(X_3) v_k &= (x_3 A^{2k}+ x_3^{-1} A^{-2k}) v_k. \qedhere
\end{align*}
\end{lem}

This proves that, up to isomorphism, the representation $\rho$ is completely determined by the numbers $t_1$, $t_2$, $x_3$ and $p$. Since $x_3$ was chosen as an arbitrary number such that $x_3^N + x_3^{-N}=t_3$, it follows that $\rho$ is completely determined by $t_1$, $t_2$, $t_3$, $p$, namely by its classical shadow $r\in \RR(S_{1,1})$ and its puncture invariant $p$.

This concludes the proof of Theorem~\ref{thm:ReconstructPuncTorus}, and therefore of the Uniqueness Theorem~\ref{thm:MainThmIntro} in the case of the one-puncture torus $S_{1,1}$. \qed

\subsection{A more general statement}

The following statement slightly improves Theorem~\ref{thm:ReconstructPuncTorus}, by providing fewer exceptions.

\begin{thm}
\label{thm:ReconstructPuncTorusBis}

Consider a character $r\in \RR(S_{1,1})$. Setting   $t_i =- \Tr r(X_i)$ for the generators $X_i$ of Proposition~{\upshape \ref{prop:PresentationPuncTorus}}, suppose that at least one of the following conditions fails:
\begin{enumerate}
\item $t_i = \pm2$ for each $i=1$, $2$, $3$;
\item $t_i = 0$ for each $i=1$, $2$, $3$;
\item one trace  $t_i$ is equal to $\pm2$,  the other two  $t_j$ are equal to $\pm \frac2{\sqrt3}\mathrm i$, and the signs are such that $t_1t_2t_3 = -\frac83$. 
\end{enumerate}
Then, up to isomorphism,  there exists a unique representation $\rho \colon \SSS(S_{1,1}) \to \End(V)$ with classical shadow $r$ and puncture invariant $p$ for every $p\in \mathbb C$ with 
$$
T_N(p) = - t_1 t_2 t_3 - t_1^2 - t_2^2 - t_3^2 + 2.
$$
\end{thm}

\begin{proof}
An easy  case-by-case analysis show that, if the hypotheses of Theorem~\ref{thm:ReconstructPuncTorus} hold, it is possible to cyclically reindex the curves $X_1$, $X_2$, $X_3$ (which does not change their relations) so that $t_3\neq \pm2$ and $t_1t_2t_3 + t_1^2 +t_2^2 \neq 0$, namely so that the hypotheses of Theorem~\ref{thm:ReconstructPuncTorus} are satisfied. Theorem~\ref{thm:ReconstructPuncTorus} then proves the uniqueness of the representation $\rho$.

To prove the existence, one can rely on \cite{BonWon2, BonWon3} or check by brute force computation that the operators explicitly given in  Lemma~\ref{lem:PunctTorusExplicitFormula} really satisfy the relations of Proposition~\ref{prop:PresentationPuncTorus}. 
\end{proof}

\section{The unpunctured torus $S_{1,0}$}
\label{sect:NoPuncTorus}

For the unpunctured torus $S_{1,0}$, Bullock and Przytycki show that the skein algebra $\SSS(S_{1,0})$ admits a presentation with the same generators and relations as $\SSS(S_{1,1})$ in Proposition~\ref{prop:PresentationPuncTorus}, but with the additional relation that $P+A^2+A^{-2}=0$. As a consequence, an irreducible representation $\rho \colon \SSS(S_{1,0}) \to \End (V)$  is equivalent to the data of an irreducible representation $\rho \colon \SSS(S_{1,0}) \to \End (V)$  with puncture invariant $p=-A^2 -A^{-2}$. The Generic Uniqueness Theorem~\ref{thm:MainThmIntro} then follows from the case of the one-puncture torus, as proved by Theorem~\ref{thm:ReconstructPuncTorus}.  

\section{The four-puncture sphere $S_{0,4}$}
\label{sect:4PuncSphere}

\subsection{A presentation for the skein algebra $\SSS(S_{0,4})$} A presentation of the skein algebra $\SSS(S_{0,4})$ of the four-puncture sphere $S_{0,4}$ can again be found in Bullock-Przytycki \cite{bullockp}.

\begin{figure}[htbp]

\SetLabels
( .21* .75) $X_1 $ \\
( .78* .75) $X_2 $ \\
( .15*.3 ) $ X_3$ \\
(.7 * .34) $P _1$ \\
( .87*.34 ) $P _3$ \\
( .7* .1) $ P_0$ \\
(.87 *.1 ) $ P_2$ \\
\endSetLabels
\centerline{\AffixLabels{\includegraphics{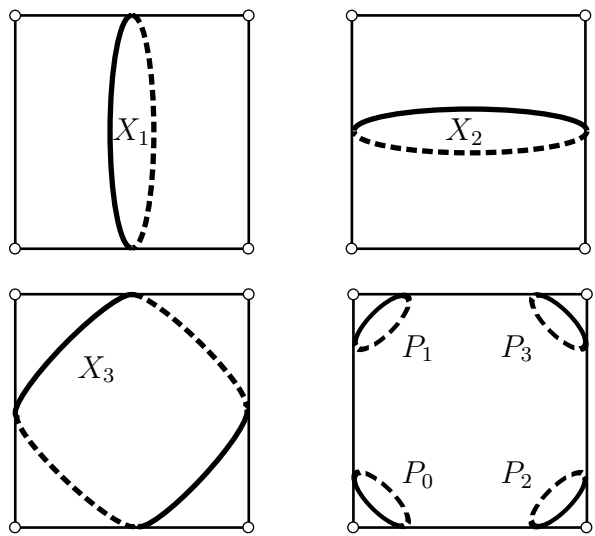}}}

\caption{Curves on the four-puncture sphere}
\label{fig:FourPuncSphere}
\end{figure}

Consider the sphere as a ``pillowcase'' obtained from a rectangle $[0,2] \times [0,1]$ by identifying each point $(0,y)$ to $(1,y)$, each $(x, 0)$ to $(2-x, 0)$, and each $(x, 1)$ to $(2-x, 1)$. Identify the four-puncture sphere $S_{0,4}$ to the surface obtained from this pillowcase by removing the four points that are the images of the six points of $[0,2] \times [0,1]$ with integer coordinates. 

This enables us to single out several simple closed curves in $S_{0,4}$. The first curve $X_1$ in the image of the two vertical arcs $\{ \frac12, \frac32 \} \times [0,1]$. The second curve $X_3$ is the image of the horizontal line segment $[0,1] \times \{ \frac12\}$. The third curve is the image of three slope 1 segments, respectively parametrized by $t\mapsto (\frac12t , \frac12t + \frac12)$, $t\mapsto (t+\frac12 , t )$ and $t\mapsto (\frac12t + \frac32, \frac12t)$ for $0\leq t\leq1$.

We also consider small loops $P_0$, $P_1$, $P_2$, $P_3$ going around the punctures, indexed in such a way that for $i=1$, $2$, $3$, the closed curve $X_i$ separates $P_0$ and $P_i$ from the other two puncture loops. See Figure~\ref{fig:FourPuncSphere}. 

We consider the elements $X_1$, $X_2$, $X_3$, $P_0$, $P_1$, $P_2$, $P_3 \in \SSS(S_{0,4})$ represented by these simple closed curves, endowed with the vertical framing. 

\begin{prop}[\cite{bullockp}]
\label{prop:PresentationFourPuncSphere}
The skein algebra $\SSS(S_{0,4})$ of the four-puncture sphere admits a presentation with generators $X_1$, $X_2$, $X_3$, $P_0$, $P_1$, $P_2$, $P_3$ as above, and with the following relations:
\begin{enumerate}
\item the $P_i$ are central;
\item $ A^2 X_1 X_{2}-A^{-2} X_{2} X_{1} =(A^4-A^{-4})X_{3}+(A^2-A^{-2})(P_0 P_3+P_1 P_2)$
 \item $   A^2 X_2 X_{3}-A^{-2} X_{3} X_{2} =(A^4-A^{-4})X_{1}+(A^2-A^{-2})(P_0 P_1+P_2 P_3) $
\item $    A^2 X_3 X_{1}-A^{-2} X_{1} X_{3} =(A^4-A^{-4})X_{2}+(A^2-A^{-2})(P_0 P_2+P_1 P_3) $
\item
$
 A^2 X_1 X_2 X_3 - A^4 X_1^2 - A^{-4}X_2^2 - A^4 X_3^2 - A^2 (P_0 P_1+P_2 P_3) X_1 -A^{-2}(P_0 P_2+P_1 P_3)X_2 
    - A^2(P_0 P_3+P_1 P_2) X_3 + (A^2 + A^{-2})^2 
  =P_0 P_1 P_2 P_3 +P_0^2+P_1^2+ P_2^2+P_3^2$. 

\end{enumerate}
\end{prop}

We will take advantage of the fact that the relations (2--4) are very similar to the relations what we already encountered for the one-puncture torus.

\subsection{Reconstructing an irreducible representations of $\mathcal{S}_A(S_{0,4})$ from its invariants} 

Let $\rho \colon \SSS(S_{0,4}) \to \End(V)$ be  an irreducible representation with classical shadow  $r\in \RR(S_{0,4})$ and puncture invariants $p_0$, $p_1$, $p_2$, $p_3$. As in the proof of Theorem~\ref{thm:ReconstructPuncTorus}, we want to  show that  $\rho$ can be reconstructed from these invariants, provided that they are generic enough.

\begin{thm}
\label{thm:ReconstructPuncSphere}
Let $\rho \colon \mathcal{S}_A(S_{0,4}) \to \text{End}(V)$ be  an irreducible representation with classical shadow  $r\in \RR(S_{0,4})$ and puncture invariants $p_0$, $p_1$, $p_2$. $p_3\in \mathbb C$. Suppose that, 
\begin{align*}
&\Tr r(X_3) \neq \pm2\\
\text{and } &\Tr r(X_3) \neq T_N(r) \text{ for every solution } r \text{ of the equation}\\
&\qquad\qquad\bigl(r^2 +  p_0 p_3 r + p_0^2 + p_3^2-4 \bigr)\bigl(r^2 +  p_1 p_2 r + p_1^2 + p_2^2-4 \bigr) =0.
\end{align*}
Then, up to isomorphism, the representation $\rho \colon \mathcal{S}_A(S_{1,1}) \to \text{End}(V)$ is completely determined by $r$ and $p$. In addition, $V$ has dimension $N$.  
\end{thm}

Theorem~\ref{thm:ReconstructPuncSphere} proves the Generic Uniqueness Theorem~\ref{thm:MainThmIntro} in the case of the four-puncture sphere. 

The rest of this section is devoted to the proof of Theorem~\ref{thm:ReconstructPuncSphere}. In particular, we henceforth assume that its hypotheses hold. The proof follows the general lines of the argument used for Theorem~\ref{thm:ReconstructPuncTorus}, but most steps will be more complicated. As a consequence, our reconstruction of the representation $\rho$ from its classical shadow $r\in \RR(S_{0,4})$ and its puncture invariants $p_i$ will not be as explicit as in that earlier case.

As in the case of the one-puncture torus, set $t_i =- \Tr r(X_i)$, so that $T_N\bigl( \rho(X_i) \bigr)=t_i \Id_V$. 
Also, in view of the relations of Proposition~\ref{prop:PresentationFourPuncSphere}, it is convenient to introduce 
\begin{align*}
q_1 & = p_0 p_1+p_2 p_3 \\
q_2 & = p_0 p_2+p_1 p_3 \\
q_3 & = p_0 p_3+p_1 p_2 \\
\Delta & = p_0 p_1 p_2 p_3 +p_0^2+p_1^2+ p_2^2+p_3^2 .
\end{align*}

Choose $x_3 \in \mathbb{C}$, such that $t_3 = x_3^N + x_3^{-N}$. Since $T_N\bigl( \rho (X_3) \bigr) = t_3 \Id_V$,  all possible eigenvalues of $\rho(X_3)$ are of the form $\lambda_k=x_3 A^{4k}+ x_3^{-1} A^{-4k}$ for $k=1$, $2$, \dots, $N-1$. These $N$ numbers $\lambda_k$ are distinct by our hypothesis that  $t_3\ne \pm 2$. As before, set  $V_k=\{v\in V; \rho(X_3)v = \lambda_k v \}$. 

The ``up'' and ``down'' operators  $\mathcal{U}_k$ and $\mathcal{D}_k$ are now given by more complicated formulas. 
\begin{align*}
\mathcal{U}_k & =A^2 \rho(X_1)-x_3 A^{4k}\rho(X_2)+\frac{q_2+x_3 A^{4k+2}q_1}{x_3 A^{4k+2}-x_3^{-1} A^{-4k-2}}\id_V\\
& =A^2 \rho(X_1)-x_3 A^{4k}\rho(X_2)+\beta^+_k \id_V\\
\mathcal{D}_k & =A^2 \rho(X_1)-x_3^{-1} A^{-4k}\rho(X_2)+\frac{-q_2-x_3^{-1} A^{-4k+2}q_1}{x_3 A^{4k-2}-x_3^{-1} A^{-4k+2}}\id_V\\
& =A^2 \rho(X_1)-x_3^{-1}A^{-4k}\rho(X_2)+\beta^-_k \id
\end{align*}
if we set $\beta_k^+ = \frac{q_2+x_3 A^{4k+2}q_1}{x_3 A^{4k+2}-x_3^{-1} A^{-4k-2}}$ and $\beta_k^- = \frac{-q_2-x_3^{-1} A^{-4k+2}q_1}{x_3 A^{4k-2}-x_3^{-1} A^{-4k+2}}$ to simplify further computations. 

\begin{lem}
\label{lem:ExpressRhoUpDown4PunctSPhere}
\pushQED{\qed}
For every $v_k \in V_k$, 
\begin{align*}
\rho(X_1) v_k & = 
-\frac{x_3^{-1} A^{-4k-2}}{x_3 A^{4k}-x_3^{-1} A^{-4k}}  \mathcal{U}_k v_k
+\frac{x_3 A^{4k-2}}{x_3 A^{4k}-x_3^{-1} A^{-4k}} \mathcal{D}_k v_k 
 \\
&\qquad\qquad + \frac{ x_3^{-1} A^{-4k-2}\beta_k^+ - x_3 A^{4k-2} \beta_k^- }{x_3 A^{4k}-x_3^{-1} A^{-4k}} v_k  \\
\rho(X_2) v_k & = -\frac{1}{x_3 A^{4k}-x_3^{-1} A^{-4k}} \mathcal{U}_k v_k+  \frac{1}{x_3 A^{4k}-x_3^{-1} A^{-4k}} \mathcal{D}_k v_k  \\
& \qquad \qquad+ \frac{\beta_k^+ - \beta_k^- }{x_3 A^{4k}-x_3^{-1} A^{-4k}} v_k\\
\rho(X_3) v_k &= (x_3 A^{4k} + x_3^{-1} A^{-4k}) v_k. \qedhere
\end{align*}

\end{lem}

\begin{lem}
\label{lem:UpDownFourPunctSphere}
The operator $\mathcal U_k$ sends each subspace $V_k$ to $V_{k+1}$, and $\mathcal D_k$ sends $V_k$ to $V_{k-1}$. 
\end{lem}

\begin{proof} Given  a vector $v_k \in V_k$, we  want to show that $\rho(X_3) \mathcal{U}_k v_k = \lambda_{k+1} \mathcal{U}_k v_k$. As in the proof of Lemma~\ref{lem:UpDown}, we expand \begin{align*}
\rho(X_3) \mathcal{U}_k v_k & =A^2 \rho(X_3 X_1)v_k - x_3 A^{4k} \rho(X_3 X_2) v_k +\beta^+_k \rho(X_3) v_k   \\
& = A^{-2} \rho(X_1)\rho(X_3)v_k + (A^4-A^{-4})\rho(X_2)v_k + (A^2 - A^{-2}) q_2v_k \\
&\qquad \qquad -  x_3 A^{4k+4} \rho(X_2)\rho(X_3)v_k + x_3 A^{4k+2} (A^4 - A^{-4}) \rho(X_1)v_k \\
&\qquad \qquad + x_3 A^{4k+2}(A^2 - A^{-2}) q_1v_k + \beta^+_k \rho(X_3) v_k
\end{align*}
using the relations of Proposition~\ref{prop:PresentationFourPuncSphere}.

Since $\rho(X_3) v_k = \lambda_k v_k$ and $\lambda_k=(x_3 A^{4k}+ x_3^{-1} A^{-4k})v_k$, it follows that
\begin{align*}
\rho(X_3) \mathcal{U}_k v_k & = A^{-2} \rho(X_1)\lambda_kv_k + (A^4-A^{-4})\rho(X_2) v_k+ (A^2 - A^{-2}) q_2v_k \\
&\qquad \qquad -  x_3 A^{4k+4} \rho(X_2)\lambda_kv_k + x_3 A^{4k+2} (A^4 - A^{-4}) \rho(X_1)v_k \\
&\qquad \qquad + x_3 A^{4k+2}(A^2 - A^{-2}) q_1v_k + \beta^+_k \lambda_k v_k \\
& =  (x_3 A^{4k+6} + x_3^{-1} A^{-4k-2}) \rho(X_1)v_k  - (x_3^2 A^{8k+4}+A^{-4}) \rho(X_2)v_k \\
&\qquad \qquad + \beta^+_k (x_3A^{4k+4} + x_3^{-1}A^{-4k-4} )v_k \\
&  = \lambda_{k+1} \bigl(A^2 \rho(X_1)-x_3 A^{4k}\rho(X_2)+\beta^+_k\bigr) v_k = \lambda_{k+1} \mathcal{U}_k v_k
\end{align*}
which shows that $\mathcal{U}_k v_k$ belongs to $V_{k+1}$. 

The proof that $\mathcal{D}_k$ sends $V_k$ to $V_{k-1}$ is very  similar. 
\end{proof}

We now have the equivalent of Lemma~\ref{lem:ComposeUpAndDown}. 
\begin{lem}
\label{lem:ComposeUpDownFourPuncSphere}
 For every $v_k \in V_k$, 
$$
\mathcal{D}_{k+1} \mathcal{U}_k v_k = R_k v_k
$$
where
$$
R_k=- (\Delta -2 + x_3^2 A^{8k+4}+ x_3^{-2} A^{-8k-4}+(x_3 A^{4k+2}+x_3^{-1} A^{-4k-2})q_3 - \beta^-_{k+1} \beta^+_k ) 
$$
\end{lem}

\begin{proof}
We begin by expanding
\begin{align*}
\mathcal{D}_{k+1} \mathcal{U}_k v_k & = (A^2 \rho(X_1) - x_3^{-1} A^{-4k-4} \rho(X_2)+\beta^-_{k+1}\Id_V)(A^2 \rho(X_1) - x_3 A^{4k} \rho(X_2)+\beta^+_k\Id_V)v_k \\
& = A^4 \rho(X_1^2)v_k - x_3^{-1} A^{-4k-2} \rho(X_2 X_1) v_k +A^2 \beta^-_{k+1} \rho(X_1)v_k\\
& \qquad\qquad - x_3 A^{4k+2} \rho(X_1 X_2)v_k + A^{-4} \rho(X_2^2)v_k - x_3 A^{4k} \beta^-_{k+1}\rho(X_2)v+k\\
& \qquad\qquad +A^2 \beta^+_{k} \rho(X_1)v_k - x_3^{-1} A^{-4k-4} \beta^+_{k}\rho(X_2)v_k + \beta^-_{k+1} \beta^+_kv_k
\end{align*}
Using Relation~(5) of Proposition~\ref{prop:PresentationFourPuncSphere}, 
\begin{align*}
\mathcal{D}_{k+1} \mathcal{U}_k v_k + \Delta v_k
& = A^4 \rho(X_1^2)v_k - x_3^{-1} A^{-4k-2} \rho(X_2 X_1) v_k +A^2 \beta^-_{k+1} \rho(X_1) v_k\\
& \qquad\qquad - x_3 A^{4k+2} \rho(X_1 X_2)v_k + A^{-4} \rho(X_2^2) v_k- x_3 A^{4k} \beta^-_{k+1}\rho(X_2)v_k\\
& \qquad\qquad +A^2 \beta^+_{k} \rho(X_1) v_k - x_3^{-1} A^{-4k-4} \beta^+_{k}\rho(X_2)v_k + \beta^-_{k+1} \beta^+_k v_k \\
& \qquad\qquad +A^2 \rho( X_1 X_2 X_3)v_k - A^4 \rho(X_1^2)v_k - A^{-4}\rho(X_2^2) - A^4 \rho(X_3^2)v_k\\
& \qquad\qquad  - A^2 q_1 \rho(X_1) v_k -A^{-2}q_2\rho(X_2) - A^2 q_3 \rho(X_3)v_k + (A^2 + A^{-2})^2 v_k\\
& =  - x_3^{-1} A^{-4k-2} \rho(X_2 X_1) v_k +A^2 \beta^-_{k+1} \rho(X_1) v_k - x_3 A^{4k+2} \rho(X_1 X_2) v_k \\
& \qquad\qquad - x_3 A^{4k} \beta^-_{k+1}\rho(X_2) v_k +A^2 \beta^+_{k} \rho(X_1) v_k - x_3^{-1} A^{-4k-4} \beta^+_{k}\rho(X_2) v_k \\
& \qquad\qquad + \beta^-_{k+1} \beta^+_k v_k +A^2 \rho( X_1 X_2 X_3)v_k  - A^4 \rho(X_3^2) - A^2 q_1 \rho(X_1)v_k \\
& \qquad\qquad-A^{-2}q_2\rho(X_2)v_k - A^2 q_3 \rho(X_3)v_k + (A^2 + A^{-2})^2 v_k.
\end{align*}

We simplify this complicated expression in a few steps. We first observe that all  terms with $\rho(X_1)$ and $\rho(X_2)$ cancel each other out. Indeed,
$$
A^2 \beta^-_{k+1} \rho(X_1)v_k +A^2 \beta^+_{k} \rho(X_1) v_k -  A^2 q_1 \rho(X_1)v_k = A^2(\beta^-_{k+1} + \beta^+_{k} - q_1) \rho(X_1)v_k = 0
$$
and
\begin{align*}
- x_3 A^{4k} \beta^-_{k+1}\rho(X_2)v_k&- x_3^{-1} A^{-4k-4} \beta^+_{k}\rho(X_2)v_k -A^{-2}q_2\rho(X_2)v_k \\
&= -A^{-2}(x_3 A^{4k+2} \beta^-_{k+1}+x_3^{-1} A^{-4k-2} \beta^+_{k}+q_2)\rho(X_2)v_k = 0,
\end{align*}
so that we are left with
\begin{align*}
\mathcal{D}_{k+1} \mathcal{U}_k v_k + \Delta v_k & =   - x_3^{-1} A^{-4k-2} \rho(X_2 X_1)v_k - x_3 A^{4k+2} \rho(X_1 X_2)v_k + \beta^-_{k+1} \beta^+_k v_k\\
& \qquad +A^2 \rho( X_1 X_2 X_3) v_k - A^4 \rho(X_3^2) v_k  - A^2 q_3 \rho(X_3)v_k + (A^2 + A^{-2})^2 v_k .
\end{align*}

We now use the fact that $v_k$ is an eigenvalue of $\rho(X_3)$:
$$
A^2 \rho( X_1 X_2 X_3) v_k = A^2 \rho( X_1 X_2) \lambda_k v_k = A^2 (x_3 A^{4k}+ x_3^{-1} A^{-4k}) \rho( X_1 X_2) v_k.
$$
Therefore
\begin{align*}
\mathcal{D}_{k+1} \mathcal{U}_k v_k + \Delta v_k 
& =  - x_3^{-1} A^{-4k-2} \rho(X_2 X_1)v_k + x_3^{-1} A^{-4k+2} \rho( X_1 X_2) v_k + \beta^-_{k+1} \beta^+_k v_k \\
& \qquad\qquad  - A^4 \rho(X_3^2) v_k  - A^2 q_3 \rho(X_3) v_k+ (A^2 + A^{-2})^2 v_k \\
 & = x_3^{-1} A^{-4k} (A^4-A^{-4})\rho(X_{3} )v_k +(A^2-A^{-2})q_3 v_k + \beta^-_{k+1} \beta^+_k v_k\\
& \qquad\qquad - A^4 \rho(X_3^2)v_k - A^2 q_3 \rho(X_3)v_k + (A^2 + A^{-2})^2 v_k 
\end{align*}
after using Relation~(1) of Proposition~\ref{prop:PresentationFourPuncSphere}. 
The formula of Lemma~\ref{lem:ComposeUpDownFourPuncSphere} then follows from a last application of the property that $\rho(X_3)v_k= (x_3 A^{4k} + x_3^{-1} A^{-4k})v_k$. 
\end{proof}

We now turn to the analogue of Lemma~\ref{lem:ProductUpsAndDowns}, and consider the product 
$$
\prod_{j=1}^N\mathcal{D}_{k+j} \prod_{j=1}^N\mathcal{U}_{k+N-j} =  \mathcal D_{k+1}\mathcal D_{k+2} \dots \mathcal D_{k+N-1}\mathcal D_{k+N} \mathcal U_{k+N-1}\mathcal U_{k+N-2} \dots \mathcal U_{k+1}\mathcal U_{k}.
$$
Since the output of Lemma~\ref{lem:ComposeUpDownFourPuncSphere} is more complicated than that of Lemma~\ref{lem:ComposeUpAndDown}, the computation will be more elaborate. We begin with a simple step.

\begin{lem}
\label{lem:ProductUpsAndDownsPuncSphere1}
For every $v_k \in V_k$, 
$$
\prod_{j=1}^N \mathcal{D}_{k+j} \prod_{j=1}^N \mathcal{U}_{k+l-j}\, v_k =  \prod_{j=1}^ N R_{j} \,v_k
$$
where $R_j$ is as in Lemma~{\upshape \ref{lem:ComposeUpDownFourPuncSphere}}. 
\end{lem}
\begin{proof}
As in the proof of Lemma~\ref{lem:ProductUpsAndDowns}, this follows from Lemma~\ref{lem:ComposeUpDownFourPuncSphere}. 
\end{proof}

We next tackle the product of Lemma~\ref{lem:ProductUpsAndDownsPuncSphere1}. 
\begin{lem}
\label{lem:ProductUpsAndDownsPuncSphere2}
$$
\prod_{j=1}^{N} R_j =   -\frac{ \bigl( t_3  - T_N( r_0)  \bigr) \bigl( t_3  - T_N( r_1)  \bigr)  \bigl( t_3  - T_N( r_2)  \bigr)   \bigl( t_3  - T_N( r_3)  \bigr)  }{ t_3^2-4}
$$
where $r_0$, $r_1$, $r_2$ and $r_3$ are the solutions of the  equation 
$$
\bigl(r^2 +  p_0 p_3 r + p_0^2 + p_3^2-4 \bigr)\bigl(r^2 +  p_1 p_2 r + p_1^2 + p_2^2-4 \bigr) =0.
$$
\end{lem}

\begin{proof} After substituting back the values of $\beta_j^+ = \frac{q_2+x_3 A^{4j+2}q_1}{x_3 A^{4j+2}-x_3^{-1} A^{-4j-2}}$, $\beta_j^- = \frac{-q_2-x_3^{-1} A^{-4j+2}q_1}{x_3 A^{4j-2}-x_3^{-1} A^{-4j+2}}$ and  $\Delta = p_0 p_1 p_2 p_3 +p_0^2+p_1^2+ p_2^2+p_3^2$ and expanding, 
$$
R_j=- (\Delta -2 + x_3^2 A^{8j+4}+ x_3^{-2} A^{-8j-4}+(x_3 A^{4j+2}+x_3^{-1} A^{-4j-2})q_3 - \beta^-_{j+1} \beta^+_j ) 
$$
can be factored as $R_j = -S_jS_j'/S_j''$ where
\begin{align*}
S_j & = \left( x_3^{-2} A^{-8 j-4} + p_0 p_3 x_3^{-1} A^{-4 j-2}+ \left(p_0^2 + p_3^2-2 \right)+p_0 p_3 x_3 A^{4 j+2}+ x_3^2 A^{8 j+4} \right) \\
S_j'& = \left( x_3^{-2} A^{-8 j-4} + p_1 p_2 x_3^{-1} A^{-4 j-2}+ \left(p_1^2 + p_2^2-2 \right)+p_1 p_2 x_3 A^{4 j+2}+ x_3^2 A^{8 j+4} \right)\\
\text{and } S_j''& =   \left(x_3 A^{4 j+2}-x_3^{-1}A^{-4 j-2}\right)^2. 
\end{align*}

The term $S_j$ looks nicer in terms of $u=x_3 A^{4 j+2}$. Then,
\begin{align*}
S_j &= u^{-2} + u^2 +p_0 p_3 u + p_0 p_3 u^{-1}+ p_0^2 + p_3^2-2  \\
 & =   (u+u^{-1})^2 + p_0 p_3 (u+u^{-1}) + p_0^2 + p_3^2-4 \\
 & = (u  + u^{-1}- r_0) (u  + u^{-1} - r_3)\\
 &= (r_0 - x_3 A^{4 j+2} - x_3^{-1} A^{-4 j-2}) (r_3 - x_3 A^{4 j+2} - x_3^{-1} A^{-4 j-2})
\end{align*}
where $r_0$ and $r_3$ are the solutions of the equation $r^2 +  p_0 p_3 r + p_0^2 + p_3^2-4 =0 $. Note that $r_0$ and $r_3$ do not depend on $j$. Lemma~\ref{lem:chebyshev}(b) then shows that
$$
\prod_{j-1}^N  S_j = \bigl( T_N(r_0) - x_3^N - x_3^{-N} \bigr)\bigl( T_N(r_3) - x_3^N - x_3^{-N} \bigr) 
= \bigl( T_N(r_0) - t_3\bigr)\bigl( T_N(r_3) -t_3 \bigr) .
$$

Similarly,
$$
\prod_{j-1}^N  S_j'
= \bigl( T_N(r_1) - t_3\bigr)\bigl( T_N(r_2) -t_3 \bigr) .
$$
where $r_1$ and $r_2$ are the solutions of the equation $r^2 +  p_1 p_2 r + p_1^2 + p_2^2-4 =0 $.

Finally
$$
\prod_{j-1}^N  S_j'' = \prod_{j-1}^N \left(x_3 A^{4 j+2}-x_3^{-1}A^{-4 j-2}\right)^2 = \left(x_3^N - x_3^{-N}\right)^2 = t_3^2-4,
$$
which concludes the computation. 
\end{proof}

\begin{lem}
\label{lem:NdimensionalPuncSphere}
The space $V$ has dimension $N$, and all eigenspaces $V_k$ of $\rho(X_3)$ are $1$-dimensional. 

More precisely, $V$ admits a basis $\{ v_1, v_2, \dots, v_N\}$ where each $v_k$ generates the eigenspace $V_k$, and where for some $u\neq 0$
$$
\mathcal U_k v_k =
\begin{cases}
v_{k+1} &\text{ if } 1\leq k \leq N-1\\
u v_1 &\text{ if } k=N
\end{cases}
$$
and
$$
\mathcal D_k v_k =
\begin{cases}
\frac1u R_N  v_N &\text{ if } k=1\\
R_{k-1}v_{k-1} &\text{ if } 2\leq k \leq N.
\end{cases}
$$
where $R_k$ is defined as in Lemma~{\upshape\ref{lem:ComposeUpDownFourPuncSphere}}. 
\end{lem}
\begin{proof}
By the combination of Lemmas~\ref{lem:ProductUpsAndDownsPuncSphere1} and \ref{lem:ProductUpsAndDownsPuncSphere2}, and by hypothesis on $t_3$, the product $\prod_{j=1}^N \mathcal{D}_{k+j} \prod_{j=1}^N \mathcal{U}_{k+l-j}$ is different from 0. The proof is then identical to that of Lemma~\ref{lem:Ndimensional}. 
\end{proof}

\begin{lem}
\label{lem:ComputeUPuncSphere}
The number $u$ occurring in Lemma~{\upshape \ref{lem:NdimensionalPuncSphere}} is completely determined by $x_3$, $t_1$, $t_2$ and the puncture invariants $p_i$. 
\end{lem}
\begin{proof}
Although we cannot be as specific as in the proof of Lemma~\ref{lem:Ndimensional}, we will follow a very similar argument. 

By Lemma~\ref{lem:ExpressRhoUpDown4PunctSPhere}, we can write $\rho(X_1)$ as a sum $\rho(X_1)= \mathcal U + \mathcal D + \mathcal I$ where, in the basis $\{ v_1, v_2, \dots, v_N \}$ of Lemma~\ref{lem:NdimensionalPuncSphere}, 
\begin{align*}
\mathcal U v_k & = 
-\frac{x_3^{-1} A^{-4k-2}}{x_3 A^{4k}-x_3^{-1} A^{-4k}}  \mathcal{U}_k v_k\\
\mathcal D v_k & = \frac{x_3 A^{4k-2}}{x_3 A^{4k}-x_3^{-1} A^{-4k}} \mathcal{D}_k v_k 
 \\
\mathcal I v_k&= \frac{ x_3^{-1} A^{-4k-2}\beta_k^+ - x_3 A^{4k-2} \beta_k^- }{x_3 A^{4k}-x_3^{-1} A^{-4k}} v_k  .
\end{align*}

The key observation is that, by the combination of Lemma~\ref{lem:NdimensionalPuncSphere} with the above definition,  $\mathcal U v_k = a_k v_{k+1}$ where $a_k$ depends only on $x_3$ if $1\leq k\leq N-1$, and where $a_N$ is $u$ times a quantity depending only on $x_3$. Similarly, $\mathcal D v_k = b_k v_{k-1}$ where $b_k$ depends only on $x_3$ and the puncture invariants $p_i$ if $2\leq k\leq N$, and where $b_1$ is $\frac1u$ times a quantity depending only on $x_3$ and the $p_i$. And $\mathcal I_k = c_k v_k$ where $c_k$ depends only on $x_3$ and the $p_i$. 
 
 If we expand $T_N\bigl( \rho(X_1) \bigr) = T_N(\mathcal U + \mathcal D + \mathcal I)$, we obtain a linear combination of terms $A_1A_2 \dots A_m$ with $m\leq N$, where each $A_i$ is equal to $\mathcal U$, $\mathcal D$ or $\mathcal I$. Since $T_N\bigl( \rho(X_1) \bigr) = t_1 \Id_V$ the only monomials that do not cancel out in this linear combination are those for which the line $A_1A_2 \dots A_{m-1}A_m V_1 $ is equal to $V_1$. 
 
 For most such monomials $A_1A_2 \dots A_m$ with a nontrivial contribution to $T_N\bigl( \rho(X_1) \bigr) $, the sequence of lines $V_1$, $A_m V_1$, $A_{m-1}A_m V_1$, \dots, $A_2 \dots A_{m-1}A_m V_1$, $A_1A_2 \dots A_{m-1}A_m V_1 =V_1$ switches as many times from $V_1$ to $V_N$ as it does from $V_N$ to $V_1$. This implies that, when we compute $A_1A_2 \dots A_m v_1$, any term $u$ is balanced by a term $\frac 1u$ and conversely, so that $A_1A_2 \dots A_m v_1$ is equal to $v_1$ times a scalar depending only on $x_3$ and the $p_i$. 
 
 Because $m\leq N$, there are exactly two exceptions to this property, namely $A_1A_2 \dots A_m = \mathcal U^N$ and $\mathcal D^N$.  These two exceptions occur with coefficient 1 in the expression of $T_N\bigl( \rho(X_1) \bigr)$ since the highest degree of $T_N(x)$ is $x^N$. Also, 
\begin{align*}
\mathcal U^N v_1 &= -u \prod_{k=1}^N \frac{x_3^{-1} A^{-4k-2}}{x_3 A^{4k}-x_3^{-1} A^{-4k}} \, v_1= -u \frac{x_3^{-N}}{x_3^N - x_3^{-N}} \,v_1\\
\mathcal D^N v_1 &=  u^{-1}  \prod_{k=1}^N \frac{x_3 A^{4k-2}}{x_3 A^{4k}-x_3^{-1} A^{-4k}} R_{k-1}\, v_1 = u^{-1} \frac{x_3^N}{x_3^N -x_3^{-N} } \prod_{k=1}^N  R_k\, v_1
\end{align*}
by Lemma~\ref{lem:NdimensionalPuncSphere}. 

Since $T_N\bigl( \rho(X_1) \bigr) = t_1 \Id_V$, the conclusion of this discussion is that
$$
t_1 =  -u \frac{x_3^{-N}}{x_3^N - x_3^{-N}} +  u^{-1} \frac{x_3^N}{x_3^N -x_3^{-N} } \prod_{k=1}^N  R_k + f(x_3, p_0, p_1, p_2, p_3)
$$
for an explicit function $f(x_3, p_0, p_1, p_2, p_3)$ of $x_3$ and of the puncture invariants $p_i$. 

The same argument applied to $T_N\bigl( \rho(X_2) \bigr) = t_2 \Id_V$ gives that
$$
t_2 =  -u \frac{1}{x_3^N - x_3^{-N}} +  u^{-1} \frac{1}{x_3^N -x_3^{-N} } \prod_{k=1}^N  R_k + g(x_3, p_0, p_1, p_2, p_3)
$$
for another function $g(x_3, p_0, p_1, p_2, p_3)$ of $x_3$ and of the puncture invariants $p_i$.

Then
\begin{equation*}
u = t_1 - x_3^N t_2 - f(x_3, p_0, p_1, p_2, p_3) + x_3^N g(x_3, p_0, p_1, p_2, p_3). \qedhere
\end{equation*}
\end{proof}

 The combination of Lemmas~\ref{lem:ExpressRhoUpDown4PunctSPhere}, \ref{lem:NdimensionalPuncSphere} and \ref{lem:ComputeUPuncSphere} then shows that, after isomorphism of $\rho$, the operators $\rho(X_1)$, $\rho(X_2)$ and $\rho(X_3)$ are completely determined by the $t_i$ and $p_j$. We are not able to give expressions as explicit as in Lemma~\ref{lem:PunctTorusExplicitFormula}, but this is enough to prove Theorem~\ref{thm:ReconstructPuncSphere}. \qed

\bibliographystyle{plain}
\bibliography{kauffman20150406}

\end{document}